\documentclass{article}
\usepackage{graphicx} 
\usepackage{hyperref}   
\hypersetup{
	colorlinks=true,
	linkcolor=blue,
	citecolor=blue,
	urlcolor=blue,
}
\usepackage{amsthm}
\usepackage{xcolor}
\usepackage{algorithm}
\usepackage{algpseudocode}
\usepackage{graphicx}
\usepackage{amsfonts}
\usepackage{amsmath,mathtools,amssymb,amsfonts}

\newcommand{\norm}[1]{\left\lVert#1\right\rVert}
\newcommand{\xddots}{%
  \raise 4pt \hbox {.}
  \mkern 6mu
  \raise 1pt \hbox {.}
  \mkern 6mu
  \raise -2pt \hbox {.}
}
\DeclareMathOperator*{\SubjectTo}{Subject\phantom{a}to:}
\DeclareMathOperator*{\Minimize}{Minimize:}

\newtheorem{assumption}{\it{Assumption}}

\newtheorem{proposition}{\bf{Proposition}}
\newtheorem{remark}{\bf{Remark}}

\title{Predictive Energy Management for Mitigating Load Altering Attacks for Islanded Microgrids Using Battery Energy Storage Systems}

\author{Satish Vedula, Koto Omiloli, Ayobami Olajube, and Olugbenga Anubi \\
    Department of Electrical and Computer Engineering\\
	Florida State University\\
	Tallahassee, Florida 32310\\
    Email: \{svedula, kao23a, aolajube, oanubi\}@fsu.edu
    }

\begin{document}

\maketitle    

\section{Abstract}
An increasing number of smart devices controlling loads opens a potential pathway for false data attacks which could alter the loads. The presence of energy storage with its ability to quickly respond to discrepancies in loads offers a promising solution for security by preventing further instabilities and potential blackouts. This paper proposes a control methodology for secure predictive energy management that uses batteries to mitigate the impact of load-altering attacks. To that extent, we develop a microgrid model along with the primary control for microgrid. The developed models and the optimization algorithm are validated through a real-time numerical simulation of a modified IEEE 9 bus system involving a battery as one of the energizing sources. The results show the effectiveness of the battery in countering the load alterations.

\textbf{Keywords: Load-altering attacks, Islanded microgrids, Battery energy storage systems, Model predictive control, and Hierarchical control.}

\section{Introduction}

Islanded microgrids (MGs) are small-scale grids that can be operated unaided by the larger main MG. Typical Islanded MGs consist of numerous smaller distributed energy resources (DERs) like fuel-operated generators and battery-based sources, including but not limited to solar, wind, and other sources. The control structure employed for the Islanded MGs is multi-layered. It is classified into primary, secondary, and tertiary control layers \cite{5546958}. The primary and secondary layers consist of droop control and grid synchronization controls. Tertiary control is responsible for energy management and maintaining the power balance in MG.  

Load Altering Attack (LAA) is a cyber-physical attack against demand response (DR) and demand side management (DSM) programs \cite{7131791}. The presence of smart devices and Internet of Things (IoT) enabled \textit{controllable power loads} presents a vulnerability in terms of exposing them to cyber-attacks such as false data injections (FDI). While existing literature \cite{10156477,9206284,10383994} provides FDI detection and mitigation mechanisms using state estimation, they are heavily limited to the power transmission side. The power measurement readings at the smart meters, which comprise a part of DSM, are tampered with, resulting in false load data being generated. LAAs are classified as static and dynamic LAAs. Static LAAs focus on altering the load by shifting the values (no attack dynamics involved). However, dynamic LAAs involve more dynamics in terms of dictating the rate of change of loads along with the static shift in the load demand. Numerous researchers have been investigating the impact, detection, and mitigation of both static and dynamic LAAs on power grids \cite{5976424,7723861,7436350,9838515,9308900,10197626,9997100,10591466,9483046}. 

Dynamic LAAs and their impact on the power grid stability were studied in \cite{7131791,7723861}. In \cite{7131791}, dynamic LAAs were first introduced as opposed to the static LAAs as discussed in  \cite{5976424}. A data-driven time-frequency analysis approach for LAA detection was presented in \cite{7436350}. LAA detection, reconstruction, and mitigation using a battery is presented in \cite{9838515}. The authors proposed a sliding mode observer and a super-twisting battery control approach for LAA mitigation. A cyber-resilient economic dispatch is proposed as a solution to tackle LAA in \cite{9997100}. However, the authors did not consider battery storage and a predictive management framework. 

\begin{figure*}[h!]
    \centering
    \includegraphics[width=0.99\textwidth]{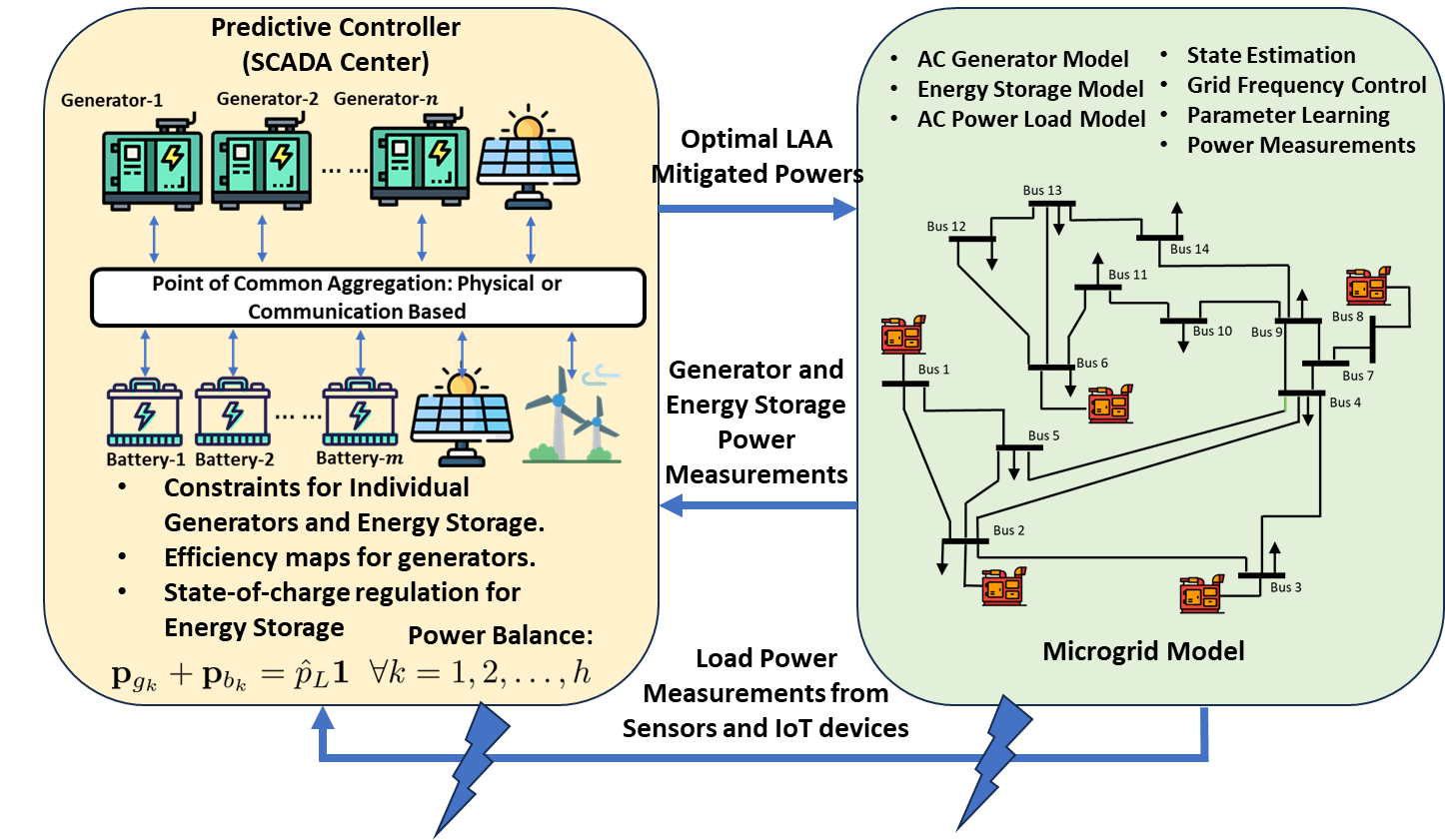}
    \caption{Proposed control methodology using predictive energy management for LAA mitigation using batteries.}
    \label{fig:overall_framework}
\end{figure*}

\textbf{The main contributions of this work are}: while the existing literature presented above focuses on LAA detection and estimation (to the best of our knowledge this is the first time a predictive energy management is used to mitigate LAAs), our method proposes the use of predictive energy management and the use of batteries to account for or mitigate the impact of both static and dynamic LAAs. We test the proposed LAA mitigating energy management framework for its real-time capabilities using a real-time simulator. The discussion presents a clear insight into the use of batteries and their role in accounting for LAAs. Fig.\ref{fig:overall_framework} shows the proposed predictive energy management framework using batteries to mitigate LAAs. The proposed framework consists of a hierarchical control structure with the primary control addressing the frequency, voltage, and current control in the MG and the secondary level addressing the energy management problem.

The rest of the paper is organized as follows: notations and required preliminaries are presented in Section \ref{sec:notations_laa}. In Section \ref{sec:mg_model_laa} we present the modeling and the control (primary) for the AC MG consisting of an AC generator, AC power loads, and DC batteries, and tailor the models to present a unified AC islanded MG. Sections \ref{sec:energymanagement_laa} and \ref{sec:realtime_sim_laa} present the proposed LAA mitigating predictive energy management strategy using batteries and a real-time numerical simulation for an AC islanded MG.

\section{Notations and Preliminaries}\label{sec:notations_laa}
The set of real numbers, and a vector containing $n$ real numbers are denoted as $\mathbb{R}$, and $\mathbb{R}^n$. Lower-case alphabets represent scalars in real space (i.e. $x \in \mathbb{R}$) and bold alphabets represent vectors in real space (i.e $\mathbf{x} \in \mathbb{R}^n$). A bold alphabet with a subscript represents the index number (i.e $\mathbf{x}_j \in \mathbb{R}, j \in \{1,2,\ldots,n\}$ is the $j^{th}$ element among $n$ elements). $\mathbf{1}$, and $\underline{\mathbf{0}}$ denotes the vector of ones and zeros. $X \in \mathbb{R}^{a \times b}$ denotes a real matrix with $a$ rows and $b$ columns. $X^\top$ denotes the transpose of the matrix $X$. The space of bounded and square-integrable functions is denoted by $\mathcal{L}_{\infty}$, and $\mathcal{L}_2$. For a vector $\mathbf{x} \in \mathbb{R}^n$, we denote the 2-norm, and the 1-norm respectively as  $\|\mathbf{x}\|_2 \triangleq \sqrt{\mathbf{x}^\top\mathbf{x}}$ and $\|\mathbf{x}\|_1
 \triangleq \sum_{i=1}^{n}|\mathbf{x}_i|$. The dot product/inner product of the two vectors $\mathbf{x} \text{ and } \mathbf{y} \in \mathbb{R}^n$ is denoted as $\mathbf{x}^\top \mathbf{y}$. If $n=2$, the cross product/outer product is denoted as $\mathbf{x}^\top J \mathbf{y}$, where $J \triangleq\begin{bmatrix}
     0 & 1 \\ -1 & 0
 \end{bmatrix}$.

\begin{figure}
    \centering
    \includegraphics[width=0.9\textwidth]{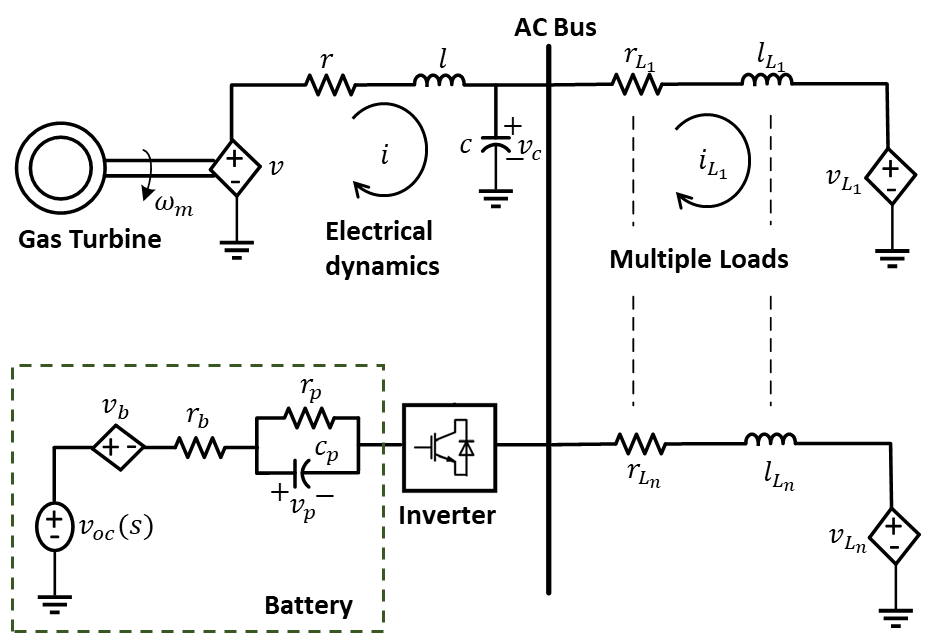}
    \caption{AC microgrid consisting of an AC generator, DC battery connected to the AC bus through an inverter, and multiple AC power loads}
    \label{fig:circuit_acgrid}
\end{figure}

\section{Islanded Microgrid}\label{sec:mg_model_laa}
In this work, we consider an AC islanded microgrid model. A typical AC microgrid consists of multiple AC generators individually driven by a turbine governor, multiple AC controllable power loads, and a few inverter-based resources such as batteries, solar photovoltaics, and wind farms. However, given the non-deterministic nature of solar and wind power generation, the power generated through them is stored in the batteries and then used for power dispatch. Most generators in islanded microgrids exhibit grid-forming capabilities by providing the necessary inertia to support the grid stability, rigidity, and overall performance\cite{du2020modeling}. The IBRs integrated at some points in such grids can either operate in the grid-following mode or grid-forming mode via intelligent control and synchronization techniques \cite{rathnayake2021grid}. Therefore, this section presents the modeling and control of AC generators, AC controllable power loads, and batteries based on Fig.~\ref{fig:circuit_acgrid}.

\subsection{Gas turbine driven AC generator modeling and control}
The AC generator model consists of a gas-turbine-driven prime mover that acts as a mechanical input to the generator. In traditional microgrids, the grid frequency is regulated by operating the rotating prime mover at a known rated speed. The model of the prime mover (mechanical), along with the electrical dynamics of a single AC generator in dq coordinates is given as
\begin{subequations}\label{ac_dynamics}
\begin{align}
    \tau \dot{\omega}_m &= -d \omega_m + T_m - T_e, \\
    l \dot{\mathbf{i}} &= -(rI - l\omega_e J)\mathbf{i} + \mathbf{v} - \mathbf{v}_c, \\
    c \dot{\mathbf{v}}_c &= \mathbf{i} - \mathbf{i}_{r} + \omega_e c J \mathbf{v}_c,
\end{align}
\end{subequations}
where $\omega_m, T_m \in \mathbb{R}$ denotes the mechanical speed and the mechanical torque of the prime mover, $\omega_e$ denotes the electrical speed, the relationship between mechanical and the electrical speed depends on the number of poles of the generator ($z$) and is given as $\omega_e = z\omega_m/2$. $T_e = \mathbf{v}_c^\top \mathbf{i}/\omega_e$ denotes the electrical torque of the generator. $\mathbf{i}, \mathbf{v}_c \in \mathbb{R}^2$ denotes the generator current and the capacitor's terminal voltage. $\mathbf{v} \in \mathbb{R}^2$ denotes the controllable voltage of the generator (control input). $\mathbf{i}_{r} \in \mathbb{R}^2$ denotes the reference current (usually determined by the load current, more on this will be discussed/presented in subsequent developments). $\tau, d, l, r, c \in \mathbb{R}$ denote the moment of inertia (\textsf{kg-m$^2$}), damping, inductance (\textsf{Henry}), resistance (\textsf{Ohm}), and terminal capacitance (\textsf{Farad}) of the mechanical prime mover and electrical generator circuit respectively. Two control objectives (one mechanical and one electrical): to regulate the prime mover speed $\omega_m$ to the synchronous speed for \eqref{ac_dynamics}a (for example, in the United States the operating frequency is 60 \textsf{Hz}, so the reference mechanical speed is generated as $\omega_{m_r} = 4\pi60/z$), and to regulate the terminal capacitance-voltage $\mathbf{v}_c$ to a known nominal reference voltage $\mathbf{v}_{{c}_r}$ for \eqref{ac_dynamics}b-c are designed and analyzed next. 
\begin{assumption}
    The prime mover dynamics in \eqref{ac_dynamics}a are dependent on the electrical torque $T_e$, which is a function of the capacitor voltage $\mathbf{v}_c$ and generator current $\mathbf{i}$. It is a well-known fact that electrical dynamics are fast compared to mechanical dynamics. Therefore we assume that $T_e = \mathbf{v}_{c_r}^\top \mathbf{i}/\omega_e$.
\end{assumption}

For regulating the prime mover speed to the synchronous speed, we employ a PI controller, and it is designed as follows
\begin{align*}
    T_m = T_e + k_P (\omega_m(t) - \omega_{m_r}) + k_I \int_{0}^{\infty} (\omega_m(\nu) - \omega_{m_r})d\nu,
\end{align*}
where $k_P, k_I$ are the proportional and integral gains. Next, we develop a parameter learning controller for the electrical dynamics of the generator. Consider the terminal voltage control error as
\begin{align*}
    \mathbf{\tilde{v}}_c = \mathbf{v}_c - \mathbf{v}_{c_r}.
\end{align*}
Taking the first time-derivative, multiplying by $c$ and substituting \eqref{ac_dynamics}c yields,
\begin{align*}
    c\mathbf{\dot{\tilde{v}}}_c = \mathbf{i} - \mathbf{i}_{r} + \omega_e cJ\mathbf{v}_c.
\end{align*}
Next, for some $\alpha > 0$ consider filtered error dynamics as
\begin{align}\label{filtered_error_dynamics}
    \boldsymbol{\eta} = \mathbf{\dot{\tilde{v}}}_c + \alpha \mathbf{\tilde{v}}_c.
\end{align}
\begin{remark}
    It is assumed that the terminal voltage control error derivative $\mathbf{\dot{\tilde{v}}}_c$ is measurable. In cases where this is not a feasible assumption, a derivative filter (with high gain) can be designed to estimate $\mathbf{\dot{\tilde{v}}}_c$ \cite{Dixon} given as.
    \begin{align*}
    \mathbf{\dot{v}}_{f} = -k_1 \mathbf{v}_{f} - k_1^2 \mathbf{\tilde{v}}_c, &\hspace{4mm}
    \mathbf{\tilde{v}}_{f} = \mathbf{v}_{f} + k_1 \mathbf{\tilde{v}}_{c},
\end{align*}
where $\mathbf{v}_f$ is an auxiliary filtered derivative state and $\mathbf{\tilde{v}}_f$ approximates $\mathbf{\dot{\tilde{v}}}_c$ for $k_1 >> 0$.
\end{remark}

Taking the first time-derivative of \eqref{filtered_error_dynamics}, multiplying with $cl$ and substituting \eqref{ac_dynamics}b-c yields the following open-loop filtered error dynamics,
\begin{align}\label{open_loop_filtered_error}
cl\boldsymbol{\dot{\eta}} = -(rI-l\omega_e J)\mathbf{i} + \omega_e J \mathbf{i} - l \omega_e  J \mathbf{i}_r - cl\omega_e^2 \mathbf{v}_c + \alpha l(\mathbf{i} - \mathbf{i}_r + c\omega_e J \mathbf{v}_c) + \mathbf{v} - \mathbf{v}_c.
\end{align}
The dynamics in \eqref{open_loop_filtered_error} can be linearly parameterized \cite{Dixon} into an unknown vector of parameters and a known matrix of measurements known as the regressor matrix. Consequently, the dynamics in \eqref{open_loop_filtered_error} are expressed as
\begin{align}\label{parameterized_open_loop}
    cl\boldsymbol{\dot{\eta}} = Y\boldsymbol{\theta} + \mathbf{v} - \mathbf{v}_c,
\end{align}
where
$$Y \hspace{-1mm} = \hspace{-1mm}\begin{bmatrix} -\mathbf{i} & 2\omega_e J\mathbf{i}-\omega_e J \mathbf{i}_r + \alpha(\mathbf{i} - \mathbf{i}_r) & -\omega_e^2\mathbf{v}_c + \alpha\omega_e J \mathbf{v}_c \end{bmatrix}$$  
$$\boldsymbol{\theta} = \begin{bmatrix} r & l & lc \end{bmatrix}^\top.$$

Next, for some $k > 0$, consider the following control law,
\begin{align}\label{control_law_acgenerator}
    \mathbf{v} = -Y\boldsymbol{\hat{\theta}} - k\boldsymbol{\eta} + \mathbf{v}_{c_r}.
\end{align}
Substituting the control law \eqref{control_law_acgenerator} in the open-loop filtered error dynamics in \eqref{parameterized_open_loop} yields the following closed-loop filtered error dynamics
\begin{align}\label{closed_loop_acdynamics}
    cl\boldsymbol{\dot{{\eta}}} = Y\underbrace{(\boldsymbol{\theta} - \boldsymbol{\hat{\theta}})}_{\boldsymbol{\tilde{\theta}}} - k\boldsymbol{\eta} - \mathbf{\tilde{v}}_c.
\end{align}
$\boldsymbol{\hat{\theta}} \in \mathbb{R}^3$ is the estimate of the parameters to be learned. Consequently, given $\gamma > 0$, the dynamic update law for the parameter estimation is 
\begin{align}\label{parameter_proj_update}
    \boldsymbol{\dot{\hat{\theta}}} = \gamma Y^\top \boldsymbol{\eta}
\end{align}

\begin{proposition}
    The control law in \eqref{control_law_acgenerator}, along with the parameter update law in \eqref{parameter_proj_update}, globally asymptotically stabilizes the origin of the filtered open-loop error dynamics in \eqref{parameterized_open_loop}.
\end{proposition}
\begin{proof}
    Consider the following Lyapunov candidate function
    \begin{align*}
        V = \frac{cl}{2}\boldsymbol{\eta}^\top \boldsymbol{\eta} + \frac{1}{2\gamma}\boldsymbol{\tilde{\theta}}^\top \boldsymbol{\tilde{\theta}} + \frac{1}{2}\tilde{\mathbf{v}}_c^\top \tilde{\mathbf{v}}_c.
    \end{align*}
Taking the first derivative along time variables yields,
\begin{align*}
    \dot{V} = \boldsymbol{\eta}^\top cl\boldsymbol{\dot{\eta}} - \frac{1}{2}\boldsymbol{\tilde{\theta}}^\top \boldsymbol{\dot{\hat{\theta}}} + \mathbf{\tilde{\mathbf{v}}}_c^\top \mathbf{\dot{\tilde{\mathbf{v}}}}_c.
\end{align*}
Substituting \eqref{closed_loop_acdynamics}, \eqref{parameter_proj_update}, and \eqref{filtered_error_dynamics} yields,
\begin{align*}
\nonumber    \dot{V} = \boldsymbol{\eta}^\top (Y\boldsymbol{\tilde{\theta}} - k\boldsymbol{\eta} - \mathbf{\tilde{v}}_c) - \frac{1}{\gamma}\boldsymbol{\tilde{\theta}}^\top \gamma Y^\top \boldsymbol{\eta} + \mathbf{\tilde{\mathbf{v}}}_c^\top(\boldsymbol{\eta} - \alpha \mathbf{\tilde{v}}_c).
\end{align*}
Rearranging yields,
\begin{align}\label{lyapunov_nsd}
    \dot{V} = -k\norm{\boldsymbol{\eta}}^2 - \alpha \norm{\mathbf{\tilde{v}}_c}^2.
\end{align}
This implies $\dot{V}$ is negative semi-definite, it follows that $V \in \mathcal{L}_{\infty}$, which implies that $\boldsymbol{\eta},\boldsymbol{\tilde{\theta}},\mathbf{\tilde{v}}_c \in \mathcal{L}_{\infty}$. Integrating \eqref{lyapunov_nsd} yields,
\begin{align*}
    V(\infty) - V(0) \leq - \int_{0}^{\infty}(k\norm{\boldsymbol{\eta}(\nu)}^2 + \alpha \norm{\mathbf{\tilde{v}}_c(\nu)}^2)d\nu,
\end{align*}
from which it follows that $\boldsymbol{\eta},\mathbf{\tilde{v}}_c \in \mathcal{L}_2$, moreover $\boldsymbol{\dot{\eta}},\mathbf{\dot{\tilde{v}}}_c \in \mathcal{L}_{\infty}$ implying uniform continuity. From Barbalat's lemma it follows that $\boldsymbol{\eta},\mathbf{\tilde{v}}_c \longrightarrow \underline{\mathbf{0}}$.
\end{proof}

\subsection{AC power load model and control}
The AC power load draws the desired current from the generators and other power sources from the grid. A load current controller ensures that the desired current is drawn from the grid. First, we present the dynamics of a single AC power load in dq coordinates as
\begin{align}\label{ac_powerload_dynamics}
    l_L \mathbf{\dot{i}}_L = -(r_LI - l_L\omega_e J)\mathbf{i}_L + \mathbf{v}_L - \mathbf{v}_g,
\end{align}
where $\mathbf{i}_L, \mathbf{v}_L \in \mathbb{R}^2$ are the load current and controlled load voltage. $\mathbf{v}_g$ denotes the grid voltage to which the load is connected. $r_L, l_L \in \mathbb{R}$ are the load resistance and the inductance respectively. Considering the grid voltage, the voltage drops across the load resistance, and load inductance as disturbances and assuming that the disturbances are upper bounded by a known constant $\overline{L}$ i.e. $-(r_LI - l_L\omega_e J)\mathbf{i}_L - \mathbf{v}_g \triangleq f_L(\mathbf{i}_L,\mathbf{v}_g)$, and $\norm{f_L(\mathbf{i}_L,\mathbf{v}_g)} < \overline{L}$, the \eqref{ac_powerload_dynamics} can be expressed as
\begin{align}\label{ac_powerload_bounded_dynamics}
     l_L \mathbf{\dot{i}}_L = f_L(\mathbf{i}_L,\mathbf{v}_g) + \mathbf{v}_L.
\end{align} 
Next, we present the current control design for the AC power load. Consider the load current error as follows
\begin{align}
    \mathbf{\tilde{i}}_L = \mathbf{i}_L - \mathbf{i}_{L_r}.
\end{align}
For some $\alpha > 0$, we define a sliding manifold as
\begin{align}
 \boldsymbol{\sigma} = \mathbf{\tilde{i}}_L.
\end{align}
Taking the derivative along the time variables and multiplying with $l_L$ yields,
\begin{align}\label{sliding_derivative_acpl}
    l_L \boldsymbol{\dot{\sigma}} = l_L\mathbf{\dot{\tilde{i}}}_L.
\end{align}
Substituting \eqref{ac_powerload_bounded_dynamics} yields,
\begin{align}\label{openloop_acpl_sliding}
    l_L \boldsymbol{\dot{\sigma}} = f_L(\mathbf{i}_L,\mathbf{v}_g) + \mathbf{v}_L,
\end{align}
Consequently, for some $k,\rho > 0$, the control law $\mathbf{v}_L$ is designed as
\begin{align}\label{ac_pl_controllaw}
    \mathbf{v}_L &= -\rho \textsf{sign}(\boldsymbol{\sigma}) - k\boldsymbol{\sigma},
\end{align}

\begin{proposition}
    The control law in \eqref{ac_pl_controllaw} stabilizes the origin of the system in \eqref{openloop_acpl_sliding} globally asymptotically. Moreover, it drives $\boldsymbol{\sigma}$ to zero asymptotically ($\norm{\boldsymbol{\sigma(0)}} \neq 0$) for some $\rho > \overline{L}$ in finite time given by 
    \begin{align}\label{slifding_reaching_cond}
        t_r = -\frac{l_L}{k}\textsf{log}\Bigg(\frac{\frac{\rho - \overline{L}}{k}}{\norm{\boldsymbol{\sigma(0)}} + \frac{\rho - \overline{L}}{k}}\Bigg) 
    \end{align}
\end{proposition}
\begin{proof}
    Consider the following Lyapunov function candidate
    \begin{align*}
        V = \frac{l_L}{2}\boldsymbol{\sigma}^\top \boldsymbol{\sigma}.
    \end{align*}
    Taking the first derivative along the time variables and substituting \eqref{openloop_acpl_sliding} yields,
    \begin{align*}
        \dot{V} = \boldsymbol{\sigma}^\top (f_L(\mathbf{i}_L,\mathbf{v}_g) + \mathbf{v}_L)
    \end{align*}
    Next, substituting the control law \eqref{ac_pl_controllaw} yields,
    \begin{align*}
        \dot{V} &= \boldsymbol{\sigma}^\top f_L(\mathbf{i}_L,\mathbf{v}_g) - \rho \boldsymbol{\sigma}^\top \textsf{sign}(\boldsymbol{\sigma}) -k \boldsymbol{\sigma}^\top \boldsymbol{\sigma}, \\
        &\leq \norm{\boldsymbol{\sigma}} \norm{f_L(\mathbf{i}_L,\mathbf{v}_g)} - \rho \norm{\boldsymbol{\sigma}} - k\norm{\boldsymbol{\sigma}}^2, \\
        &\leq -(\rho - \overline{L})\norm{\boldsymbol{\sigma}}-k\norm{\boldsymbol{\sigma}}^2
    \end{align*}
    On observation that
    \begin{align}\label{lyap_diff_inequality}
        \dot{V} \leq -(\rho - \overline{L})\sqrt{\frac{2V}{l_L}} - \frac{2k}{l_L}V.
    \end{align}
    Now let us consider the boundary of the differential inequality in \eqref{lyap_diff_inequality} and make a change of variable as $\sqrt{\frac{2V}{l_L}} = u$ and observe that $V = \frac{l_L}{2}u^2$, $\dot{V} = l_L u \frac{du}{dt}$. Then, the differential inequality in \eqref{lyap_diff_inequality} becomes a linear first-order differential equation (ODE) in $u$ as
    \begin{align}\label{lyap_linear_ineq}
        \frac{du}{dt} = -\frac{\rho - \overline{L}}{l_L} - \frac{k}{l_L}u.
    \end{align}
    The solution to the ODE in \eqref{lyap_linear_ineq} given an initial condition of $u(0)$ is
    \begin{align}
        u(t) = -\frac{\rho - \overline{L}}{k} + \bigg(u(0) + \frac{\rho - \overline{L}}{k}\bigg)e^{-\frac{k}{l_L}t}.
    \end{align}
   Reverting to the terms of $V(t)$, and at $t = t_r$ yields,
   \begin{align*}
       V(t_r) = \frac{l_L}{2} \Bigg[-\frac{\rho - \overline{L}}{k} + \bigg(u(0) + \frac{\rho - \overline{L}}{k}\bigg)e^{-\frac{k}{l_L}t_r}\bigg]^2.
   \end{align*}
   At the reaching time $t_r$, $V(t_r) = 0$, yields the reaching condition in \eqref{slifding_reaching_cond}  
   \begin{align}
        t_r = -\frac{l_L}{k}\textsf{log}\Bigg(\frac{\frac{\rho - \overline{L}}{k}}{\norm{\boldsymbol{\sigma(0)}} + \frac{\rho - \overline{L}}{k}}\Bigg).
   \end{align}
\end{proof}
\begin{remark}
    It is evident from the observation that the logarithmic term $\frac{\rho - \overline{L}}{k}/{\norm{\boldsymbol{\sigma(0)}} + \frac{\rho - \overline{L}}{k}}$ is always between 0 and 1 for the conditions in the theorem statement.
\end{remark}

\subsection{Battery model and state of charge estimation}
Thevinin's model captures the battery operation's internal resistance, controllable voltage, and polarization effects. The dynamics of the battery model (Li-ion considered in this case) along with the state of charge are
 \begin{align}
    v_t &= v_{oc}(s) - v_r - v_p - v_b, \label{dynamics1}\\
    c_p \dot{v}_p &= i_b - \frac{v_p}{r_p}, \label{dynamics2} \\
    \dot{s} &= -\frac{i_b}{3600 Q_b} \label{dynamics3}
 \end{align}
 where $s \in \mathbb{R}$ is the state-of-charge (SoC) of the battery. $v_t, v_r, v_{oc}, v_p, v_b, i_b \in \mathbb{R}$ are the battery terminal voltage, voltage across internal resistance, open-circuit voltage (OCV), polarization voltage, the controllable battery voltage, and the battery current. The battery current is controlled via a controllable battery voltage $v_b$ and is given as 
 \begin{equation}
     i_b = \frac{v_t - v_b - v_{oc} - v_p}{r_b}
 \end{equation}
 $r_b, r_p, c_p, Q_b$ denote the internal resistance, polarization resistance in \textsf{Ohms}, capacitance in \textsf{Farads}, and the battery capacity in \textsf{AHr}. The battery open circuit voltage $v_{oc}$ is a function of the battery's SoC. The relationship is modeled as a linear relationship \cite{10.1115/DSCC2013-3979}. Consequently, the OCV is presented in terms of SoC as $v_{oc} = \beta_1 s +\beta_2$ (more on the choice of constants $\beta_1,\beta_2$ is presented in the numerical simulation section).
 \begin{remark}
 In this work, we assume that the battery parameters (nominal) are known. However, if the battery parameters $(r_b, r_p, c_p)$ are unknown, they can be estimated by minimizing the objective function $\sum \limits_{i=1}^{n} (\hat{v}_i - v_i) $ where $v_i$ is the battery's terminal voltage obtained from measurement data and $\hat{v}_i$ is the discretized terminal voltage model.
 \end{remark}

\subsection{AC microgrid}
The overall model of an islanded AC microgrid (with a single AC generator and an AC power load) is represented based on the previous developments as (\textbf{Note:} battery is a DC component integrated into the AC microgrid using inverters. The estimation for the battery is done separately).
\begin{subequations}\label{ac_coupled_mgmodel}
\begin{align}
    \tau \dot{\omega}_m &= -d \omega_m + T_m - T_e, \\
    l \dot{\mathbf{i}} &= -(rI - l\omega_e J)\mathbf{i} + \mathbf{v} - \mathbf{v}_c, \\
    c \dot{\mathbf{v}}_c &= \mathbf{i} - \mathbf{i}_{L} + \omega_e c J \mathbf{v}_c, \\
    l_L \dot{\mathbf{i}}_L &= -(r_LI - l_L\omega_e J)\mathbf{i}_L + \mathbf{v}_L - \
    \mathbf{v}_c.
\end{align}
\end{subequations}

    In the traditional MG, the grid voltage, load, and generator currents can be measured utilizing the sensors embedded into the MG. We assume that the currents and the voltages are measurable.

\begin{remark}
    If there are multiple AC generators and multiple AC power loads, it is straightforward to expand the model in \eqref{ac_coupled_mgmodel} to the required components. 
\end{remark}

Consequently, the active power measurements at the bus for the AC generator, the AC power load, and the battery are expressed as
\begin{align*}
    \hat{p}_g &= \frac{1}{2}\mathbf{{v}}_c^\top \mathbf{{i}}, \hspace{5mm} \hat{p}_b = v_c i_b, \hspace{5mm} \hat{p}_L = \frac{1}{2}\mathbf{{v}}_c^\top \mathbf{{i}}_L, 
\end{align*}
\textbf{Note:} $v_c \in \mathbb{R}$ is the DC voltage obtained from the AC grid voltage $\mathbf{v}_c \in \mathbb{R}^2$ as a result of the inverter/converter mechanism. 

The estimated load active powers, which are the load power measurements, are altered, and false data is injected resulting in
\begin{align*}
    \hat{p}_L = \frac{1}{2} \mathbf{{v}}_c^\top \mathbf{{i}}_L + w_L,
\end{align*}
where $w_L$ is the injected false load altering data (the choice of $w_L$ is discussed in detail in the numerical simulation section). Consequently, the power flow in the AC microgrid (with the battery acting as a supplemental support) can be expressed as
\begin{align}\label{equality_constraint}
    \hat{p}_g + \hat{p}_L + \hat{p}_b = 0.
\end{align}

Now that we have presented the primary control development, including the estimated power measurements for the elements in the MG, we next show the proposed development of a predictive energy management algorithm using the battery as a mitigation strategy for LAA. For the EM layer, we solely consider the optimal power sharing problem acting as a power reference for the designed primary controllers.

\section{Energy Management Using Battery Storage For LAA Mitigation}\label{sec:energymanagement_laa}
Consider the MPC problem of the form \cite{9658612}:
\begin{equation}\label{MPC_main}
\begin{aligned}
\Minimize_{\mathbf{p}_{i},\mathbf{s}} \quad & \sum_{{i}=1}^{n}{C}_{{i}}(\mathbf{p}_{{i}})\\
\SubjectTo \quad & \sum_{{i}=1}^{n}\mathbf{p}_{{i}} = \sum_{j=1}^{n_L}\hat{p}_L\mathbf{1}, \hspace{1mm} \forall  k = 1,\ldots,h,\\
\quad & \mathbf{s}_{k+1} = \mathbf{s}_k - \frac{T_s}{Q_b v_c}\mathbf{p}_i, \hspace{1mm} \forall  k = 1,\ldots,h, \\
\quad & \mathbf{p}_0 = \mathbf{\hat{p}}_i, \hspace{1mm} \forall k = 1,\ldots,h, \\
\quad & \underline{\mathbf{p}} \preceq \mathbf{p}_{{i}} \preceq \overline{\mathbf{p}}, \hspace{1mm} \forall  k = 1,\ldots,h,\\
& \left|\mathbf{p}_{{ik}}-\mathbf{p}_{{ik-1}} \right| \preceq {r}\mathbf{1}, \hspace{1mm} \forall  k = 1,\ldots,h, \\ 
& \underline{\mathbf{s}} \preceq \mathbf{s}_{{ik}} \preceq \overline{\mathbf{s}}, \hspace{1mm} \forall  k = 1,\ldots,h,
\end{aligned}
\end{equation}
where $i = 1,\ldots,n$ indicates the number of power source, $j = 1,\ldots,n_L$ indicates the number of loads. $k \triangleq [1,2,\ldots,h] \in \mathbb{R}^h$ is the length of the prediction horizon, $\mathbf{p}_{i} \triangleq [{\mathbf{p}_{i}}_k,{\mathbf{p}_{i}}_{k+1},\hdots,{\mathbf{p}_{i}}_{k+h-1}]^{\top}\in\mathbb{R}^h$ is the power profile for power sources $i$ over the prediction horizon of length $h$, and $\hat{p}_L\mathbf{1} \in\mathbb{R}^h$ is the desired total power held constant over the prediction horizon. $\mathbf{p}_0 \in \mathbb{R}^h$ is the initial value for the optimization and $\mathbf{\hat{p}}_i \triangleq [\hat{p}_i,0,\ldots,0]^\top \in \mathbb{R}^h$ is the estimated power measurements (for example: measured generator and battery powers). The operation cost ${C}_{i}:\mathbb{R}^{{h}}\longrightarrow\mathbb{R}_+$ is designed to capture the operating costs, for example, an efficiency map or operating at a desired rated power for an AC generator and for the battery its health monitoring and degradation management (more on the choice of $C_i$ is discussed in the numerical simulation section). ${h} \in \mathbb{N}$ represents the prediction horizon. $r \in \mathbb{R}_+$ is the ramping capability of the power source. $T_s$ is the simulation time-step. $\underline{\mathbf{p}}$ and $\overline{\mathbf{p}}$ are the lower and the upper power limitations on the power sources. $\underline{\mathbf{s}}$ and $\overline{\mathbf{s}}$ represent the lower and the upper limitations on the state of charge of the battery. The equality constraints stem from \eqref{equality_constraint} and \eqref{dynamics3}, indicating the power balance constraint and the state of charge dynamics. The optimization variables are optimized for the entire length of the horizon and the first value of the sequence is chosen as the control input. The constraint set is a polytope, which is a well-studied set, and existing methods in the literature have extensively studied the feasibility and stability of the optimization problem in \eqref{MPC_main} \cite{boyd_vandenberghe_2004}. 

\section{Real-Time Numerical Simulation}\label{sec:realtime_sim_laa}
The numerical simulations are carried out on a real-time simulator connected to a Host PC through a LAN connection (1000 \textsf{MBPS}). The computational capabilities of the HOST PC are an Intel Core i9 3.20GHz 24-core processor with 64GB of RAM. The computational capabilities of the real-time target machine are an Intel Core 3.6GHz 8-core processor with 3GB RAM. Fig.~\ref{fig:rt_simulation_setup} shows the real-time interconnection between the Host PC and the real-time simulator. The simulation is run at a fixed time-step of $1\textsf{msec}$. The optimization updates every $1\textsf{sec}$.

\begin{figure}[h!]
    \centering
    \includegraphics[width=0.9\textwidth]{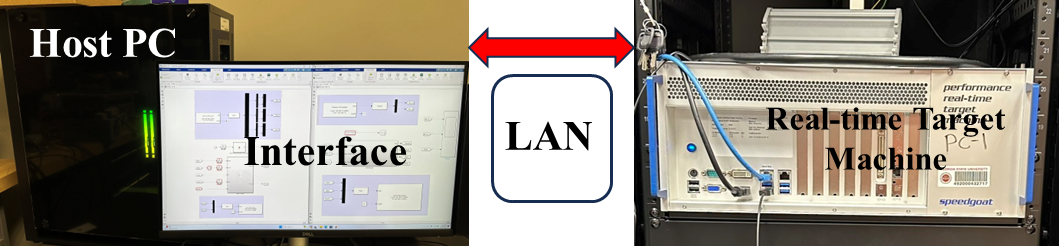}
    \caption{Numerical simulation: real-time target setup and implementation schematic.}
    \label{fig:rt_simulation_setup}
\end{figure}

\begin{figure}[h!]
    \centering
    \includegraphics[width=0.9\textwidth]{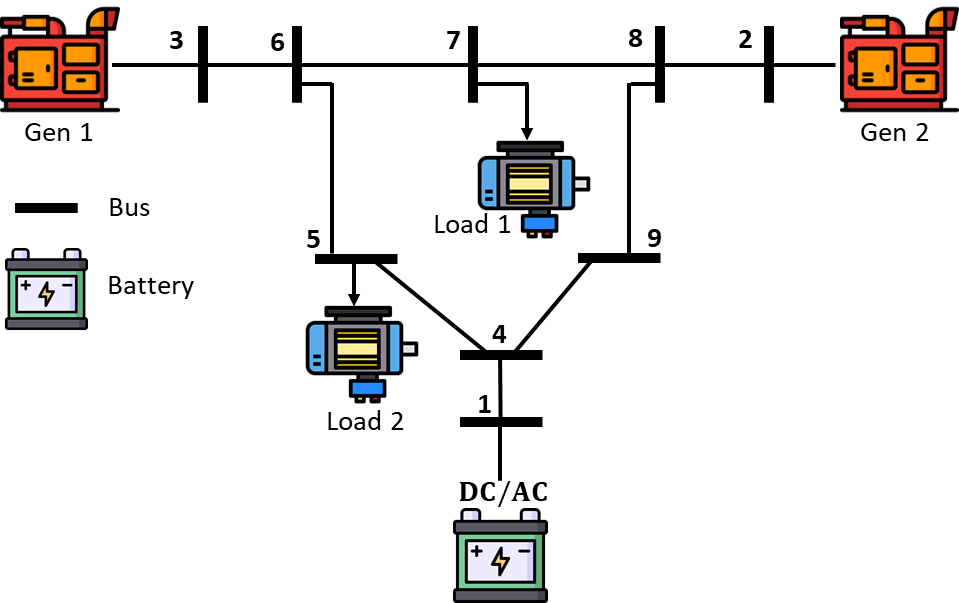}
    \caption{Modified IEEE 9 bus system used as a numerical example to test the proposed algorithm}
    \label{fig:ieee_9bus_modified}
\end{figure}

For the purpose of the simulation, we consider a modified IEEE 9 bus system consisting of two AC power loads, two AC generators, and a battery connected to the grid by means of an inverter. Fig.~\ref{fig:ieee_9bus_modified} shows the interconnection between the loads and power sources. The following are the values of the parameters for the simulation 1) for AC generator (estimated parameters) and the prime mover: $r = 0.2 \Omega$, $l = 0.03 \textsf{H}$, $c = 10\mu\textsf{F}$, $f = 60\textsf{Hz}$ (equivalent to the mechanical speed of $94.4\textsf{rad/s}$), $\tau = 2.5\textsf{kg-m$^2$}$, $d = 0.3$, 2) for the AC power load: $r_L = 0.3\Omega$, $l_L = 0.03\textsf{H}$, 3) for the battery $r_b = 0.3 \Omega$, $r_p = 0.09 \Omega$, $c_p = 10\mu\textsf{F}$. The OCV-SoC relationship is given as $v_{oc} = 1.071s + 3.357$ \textsf{kV} (linear curve-fit to the OCV-SoC data). The grid voltage is regulated to $12\textsf{kV}$. The current reference for the d-axis and the q-axis to the two AC power loads are chosen as $[6 \hspace{3mm} 0]^\top \textsf{kA}$ and $[10 \hspace{3mm} 0]^\top\textsf{kA}$. Fig.~\ref{fig:speed_voltage_track} shows the designed controller performance for the frequency regulation and the voltage regulation.

\begin{figure}
    \centering
    \includegraphics[width=0.9\textwidth]{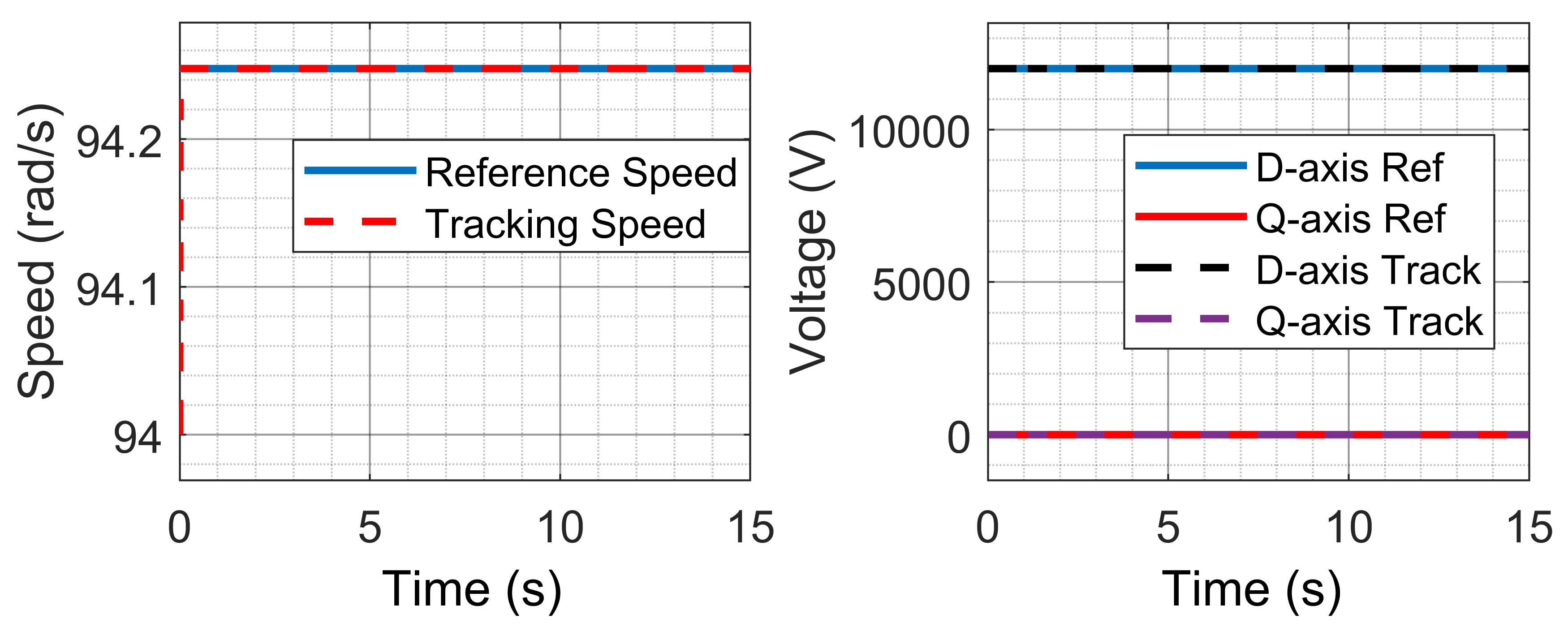}
    \caption{Tracking performance of the gas turbine speed (frequency) and the grid voltage}
    \label{fig:speed_voltage_track}
\end{figure}

\begin{figure}
    \centering
    \includegraphics[width=0.9\textwidth]{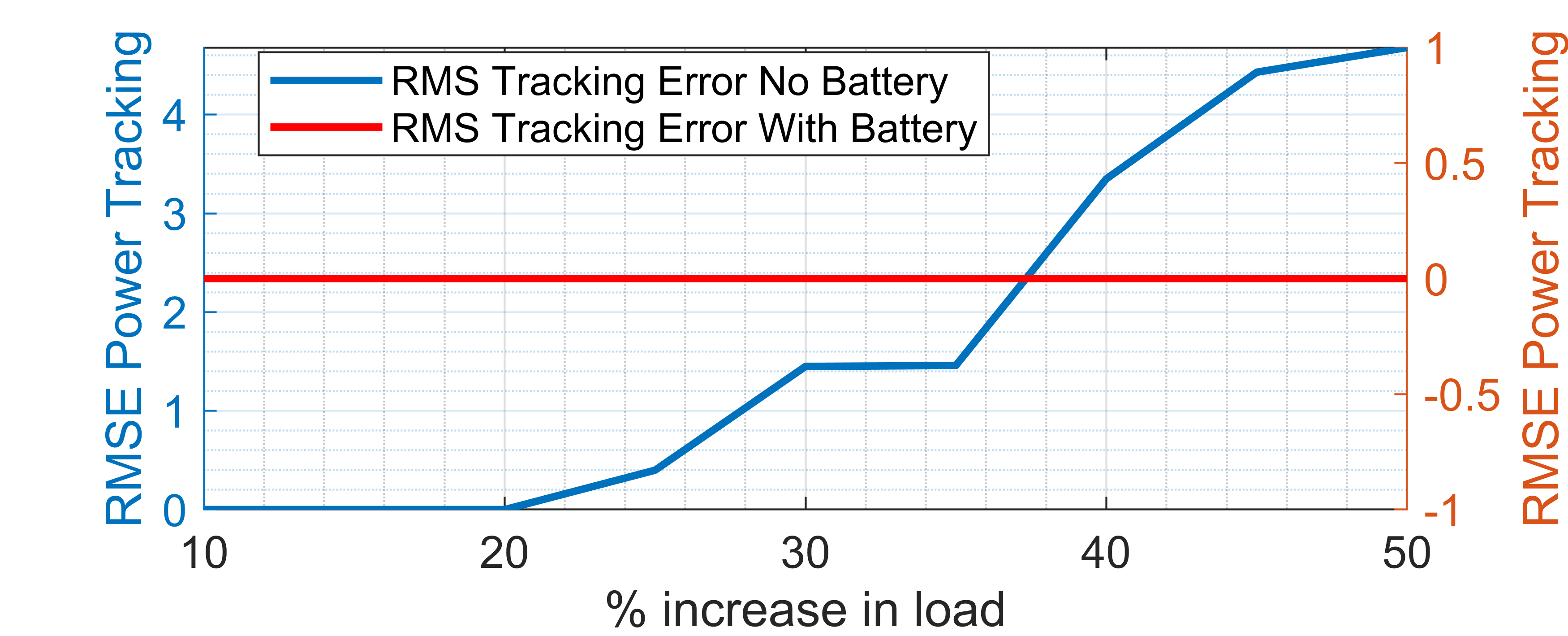}
    \caption{Root mean square power tracking error without battery on the left axis and root mean square power tracking error with the battery on the right axis versus the $\%$ alteration in the load on the X-axis.}
    \label{fig:RMSE_PT}
\end{figure}

The power ratings of the AC generators and the battery are as follows: $p_{g_1} = 24\textsf{MW}$, $p_{g_2} = 22\textsf{MW}$, $p_b = 10 \textsf{MW}$. For the optimization $n = 3$ (two generators, one battery), the operating cost in $\$/hr$ for the generators capturing the efficiency is chosen as: $C(p_{g_1}) = 0.4p_{g_1}^2 + 5.5p_{g_1} + 500$, $C(p_{g_2}) = 0.6p_{g_2}^2 + 5.3p_{g_2} + 400$. The operating cost for the battery is chosen as $C(p_b) = p_b^2$ to use the battery to negate the LAA scenarios. The optimization horizon is chosen as $h = 5\textsf{secs}$. The lower and upper power limitations on the AC generators are chosen to be $5\%$ and $95\%$ for the respective rated powers. The upper and lower power limitations and the ramping capabilities on the battery (bi-directional) are chosen as $95\%$ of the rated battery power. The negative power indicates that the battery is charging and the positive power indicates that the battery is discharging. The battery capacity $Q_b = 25\textsf{AHr}$. 

\begin{figure}
    \centering
    \includegraphics[width=0.9\textwidth]{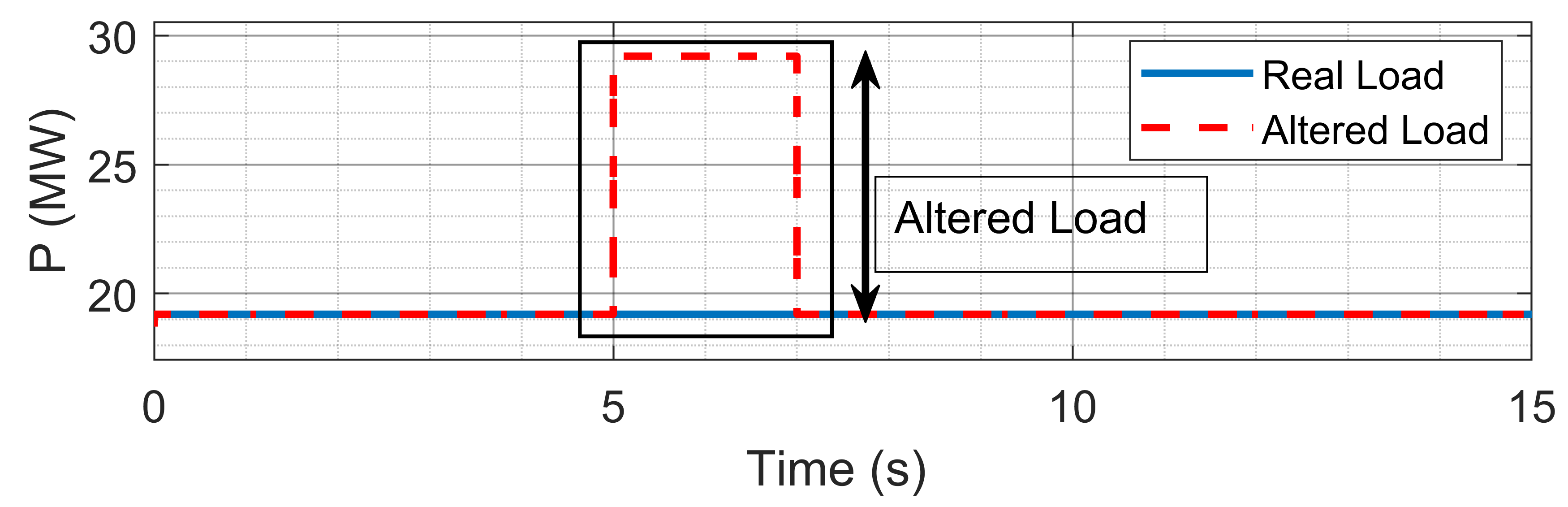}
    \caption{Load alteration}
    \label{fig:LAA_load}
\end{figure}

Fig.~\ref{fig:RMSE_PT} shows the calculated root mean square error (RMSE) for the power tracking with the increment in the LAA injection ($w_L$) from $10\%$ to $50\%$ of the measured load threshold. It can be seen that without the battery, the error increases. However, with an appropriate choice of battery sizing, the error is mitigated. The measured load is injected with a sudden ramp, altering the actual load. The spike in load injected is $50\%$ of the actual load value. The injection lasts for $1\textsf{sec}$. Fig.~\ref{fig:LAA_load} shows the actual load and the altered load in \textsf{MW}. Fig.~\ref{fig:power_no_battery} shows the impact of the injected LAA without the battery. The power sharing between the two generators, the tracking performance, and the error in the tracking active power are depicted. It can be clearly seen that there are imbalances in the power tracking.  

\begin{figure}
    \centering
    \includegraphics[width=0.9\textwidth]{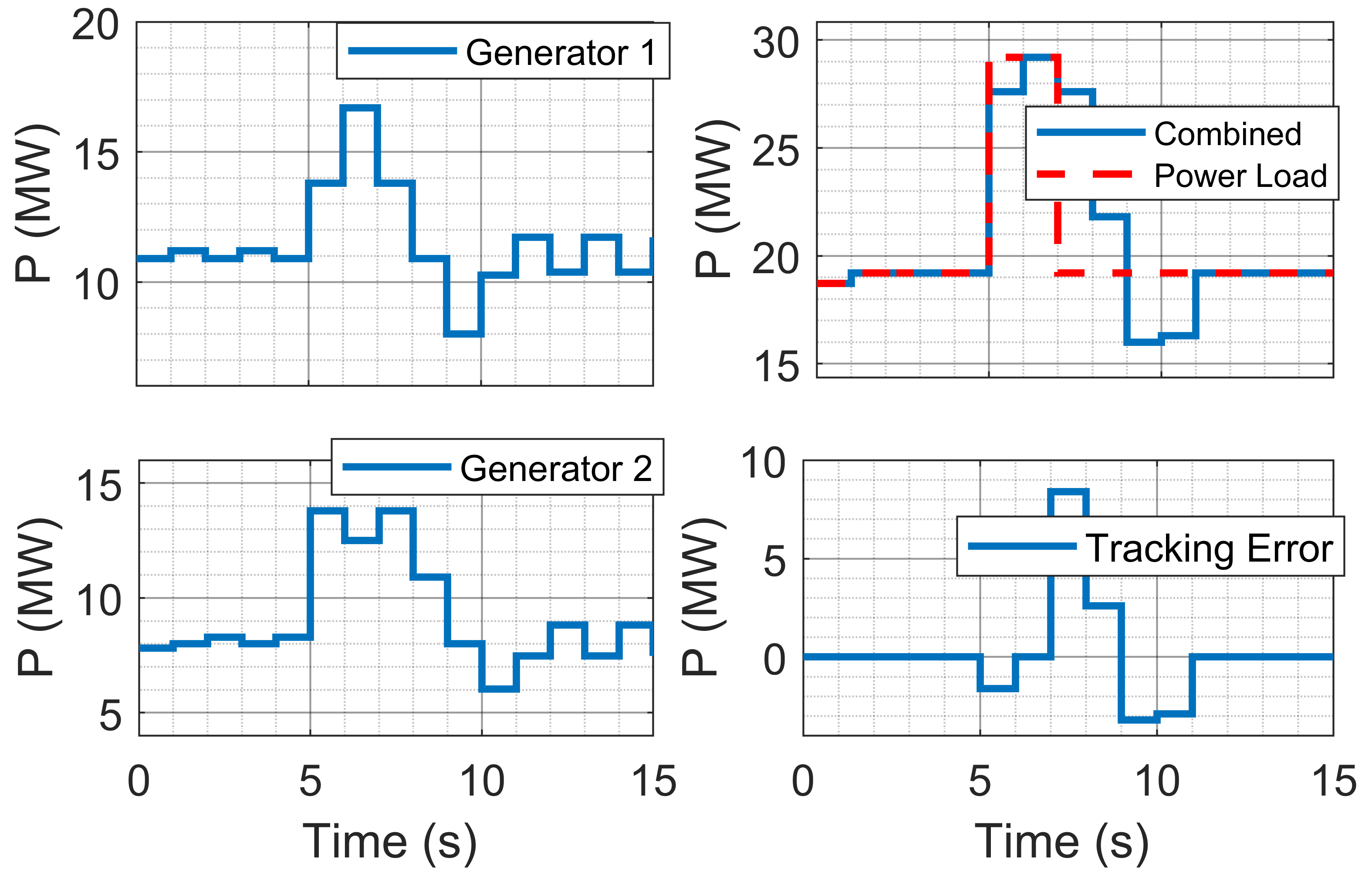}
    \caption{Power sharing under load altering without the presence of battery}
    \label{fig:power_no_battery}
\end{figure}

\begin{figure}
    \centering
    \includegraphics[width=0.9\textwidth]{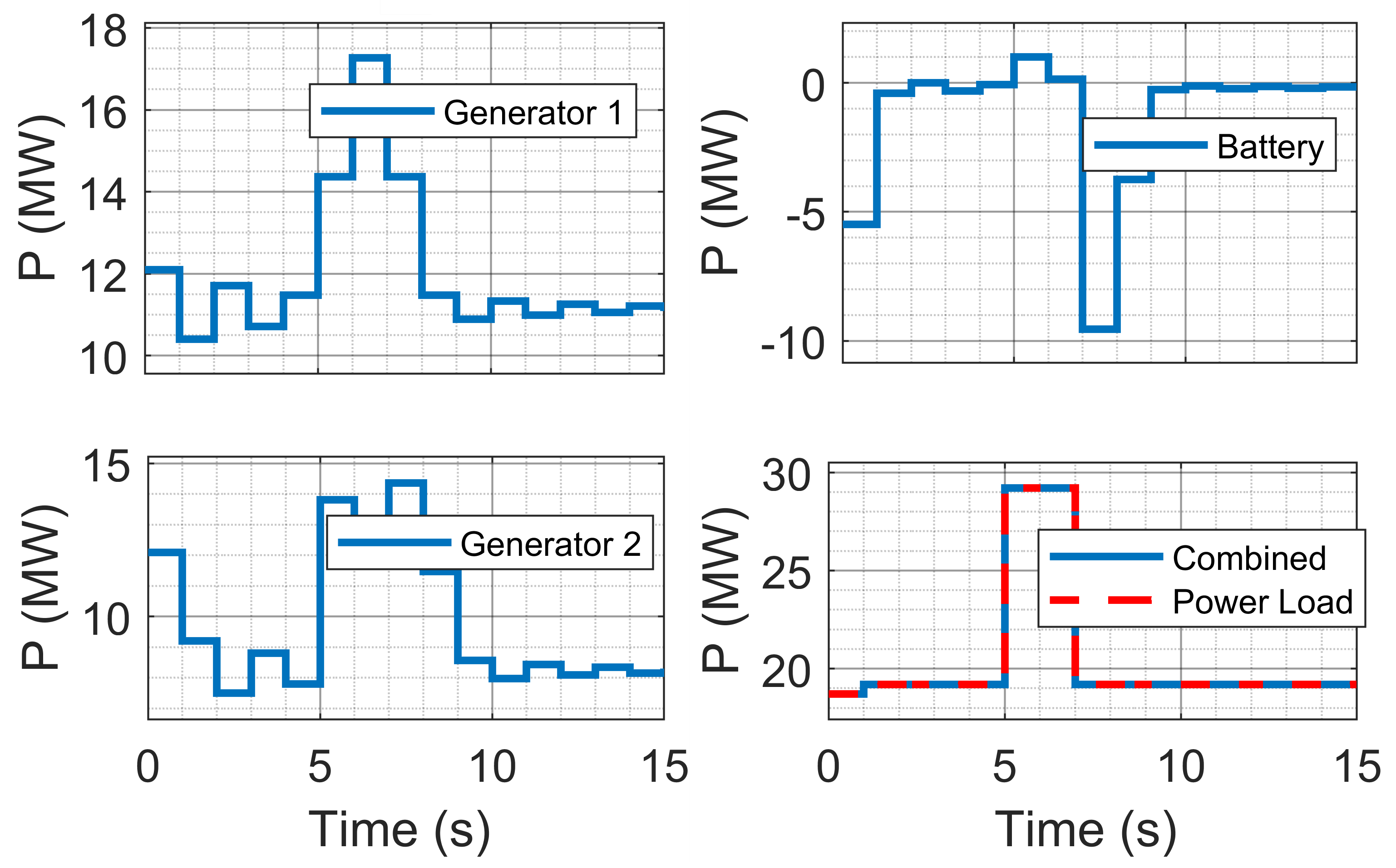}
    \caption{Power sharing under load altering with the presence of battery}
    \label{fig:power_battery}
\end{figure}

Fig.~\ref{fig:power_battery} shows the impact of the injected LAA with the battery. The power-sharing between the generators remains identical; however, the addition of the battery mitigates the impact of the injected LAA, ensuring the power tracking. It can be seen that the power balance constraint is satisfied, and there are no power imbalances arising from the LAA. This shows the effectiveness of the predictive energy management framework and the battery in tackling LAAs.

\section{CONCLUSION}
In this paper, a control methodology for a secure predictive energy management under load altering attacks, along with a model and the control of a microgrid is developed. The developed algorithm is validated using a real-time numerical simulation of a modified IEEE 9 bus system with a battery. Load measurements are varied, and the impact of the LAA with and without the presence of the battery element is studied. The effectiveness of the battery in mitigating the discrepancies caused by the load alteration was demonstrated. Future

\bibliographystyle{IEEEtran}
\bibliography{myreferences}

\end{document}